\documentclass[12pt]{amsart} 

\usepackage{amscd} 
\usepackage{amsmath} 
\usepackage{amssymb} 
\usepackage{amsthm}
\usepackage{bigdelim}
\usepackage{color} 
\usepackage{enumerate}
\usepackage{mathrsfs}
\usepackage{multirow}
\usepackage[all]{xy} 
\usepackage{hyperref}
\usepackage{comment}
\newtheorem{thm}{Theorem}[section]

\newtheorem{lem}[thm]{Lemma}
\newtheorem{prob}[thm]{Problem}
\theoremstyle{definition} 
\newtheorem{defn}[thm]{Definition}

\theoremstyle{remark}
\newtheorem{rem}[thm]{Remark}

\newtheorem{ques}[thm]{Question}

\newtheorem{cl}{Claim}

\newtheorem*{ack}{Acknowledgements}

\title{Positivity of extensions of vector bundles}
\author{Sho Ejiri, Osamu Fujino, and Masataka Iwai}
\address{Department of Mathematics, Graduate School of Science, Osaka Metropolitan University, Osaka City, Osaka 558-8585, Japan}
\email{shoejiri.math@gmail.com}
\address{Department of Mathematics, Graduate School of Science, Kyoto University, Kyoto 606-8502, Japan}
\email{fujino@math.kyoto-u.ac.jp}
\address{Department of Mathematics, Graduate School of Science, Osaka University, Osaka 560-0043, Japan}
\email{{masataka@math.sci.osaka-u.ac.jp}, {masataka.math@gmail.com}}


\begin{document}
\begin{abstract}
In this paper, we study when positivity conditions of vector bundles are preserved by extension. We prove that an extension of a big (resp.~pseudo-effective) line bundle by an ample (resp.~a nef) vector bundle is big (resp.~pseudo-effective). We also show that an extension of an ample line bundle by a big line bundle is not necessarily pseudo-effective. In particular, this implies that an almost nef vector bundle is not necessarily pseudo-effective. 
\end{abstract}
\maketitle
\markboth{SHO EJIRI, OSAMU FUJINO, and MASATAKA IWAI}{Positivity of extensions of line bundles}
\section{Introduction}

Several positivity conditions defined for line bundles (e.g.~ ampleness, nefness, bigness, and pseudo-effectivity), 
which play a key role in the study of algebraic varieties, 
are naturally extended to vector bundles (see Definitions~\ref{defn:ample_nef}--\ref{defn:big}). 
The importance of such extensions is represented by the study of projective varieties whose tangent bundle satisfies such a positivity condition.
Hartshorne~\cite{Har70} conjectured that the projective spaces are characterized as projective varieties having ample tangent bundle, which was solved affirmatively by Mori~\cite{Mor79}. 
Also, the geometric structure of a projective variety with nef (resp. pseudo-effective) tangent bundle was studied in~\cite{CP91, DPS94} (resp. \cite{HMI22}). 

When we consider a property of vector bundles, it is natural to ask whether or not it is preserved by extension.
\begin{prob}
\label{prob:exact}
Consider an exact sequence of vector bundles:
$$
0\to \mathcal E' \to \mathcal E \to \mathcal E'' \to 0.
$$
If $\mathcal E'$ and $\mathcal E''$ satisfy a positivity condition $($e.g.~ ampleness, nefness, bigness, or pseudo-effectivity$)$, then does $\mathcal E$ satisfy the same?
\end{prob}
This problem is known to hold affirmatively for ampleness and nefness (cf.~\cite[\S 6]{Laz04II}). 
Using such a property of nefness, Campana and Peternell~\cite{CP91} completed the classifications of smooth projective surfaces and threefolds with nef tangent bundle.
As their study implies, an affirmative answer to Problem~\ref{prob:exact} will be useful to know whether a vector bundle satisfies a positivity condition, so the problem has been expected to be solved for bigness and pseudo-effectivity (cf.~\cite[Problem~4.3]{HMI22}, \cite[Question~2.23]{FN23}).

This paper includes two theorems. 
One of them gives an affirmative and partial answer to Problem~\ref{prob:exact} for bigness and pseudo-effectivity.
\begin{thm} \label{thm:nef-psef}
Let $X$ be a normal projective variety over an algebraically closed field. 
Let $\mathcal E$ and $\mathcal G$ be vector bundles on $X$. 
Let $\mathcal L$ be a line bundle on $X$. 
Suppose that there exists the following exact sequence:
$$
0\to \mathcal G \to \mathcal E \to \mathcal L \to 0.
$$
\begin{enumerate}[$(1)$]
\item If $\mathcal G$ is nef and $\mathcal L$ is pseudo-effective, 
then $\mathcal E$ is pseudo-effective. 
\item If $\mathcal G$ is ample and $\mathcal L$ is big, 
then $\mathcal E$ is big. 
\end{enumerate}
\end{thm}

The other theorem solves negatively Problem~\ref{prob:exact} for bigness and pseudo-effectivity, and also shows that the assumption that $\mathcal G$ is nef (resp. ample) of (1) (resp. (2)) in Theorem~\ref{thm:nef-psef} cannot be weakened to that $\mathcal G$ is pseudo-effective (resp. big). 

\begin{thm} \label{thm:negative}
Let $k$ be an algebraically closed field. 
Then there exist a smooth projective surface $S$ over $k$ 
and a vector bundle $\mathcal V$ on $S$ with the following properties:
\begin{itemize}
\item there exists an exact sequence 
$$
0\to\mathcal L \to \mathcal V \to \mathcal M \to 0
$$
such that $\mathcal L$ is a big line bundle on $S$ and $\mathcal M$ is an ample line bundle on $S$;
\item $\mathcal V$ is not pseudo-effective $($so not big$)$.
\end{itemize}
\end{thm}

Since a big line bundle is weakly positive, the above theorem also tells us 
that weak positivity is not necessarily preserved by extension. 
Here, weak positivity is a notion introduced by Viehweg~\cite{Vie82} 
(see Definition~\ref{defn:psef_wp}), 
which is a stronger condition than pseudo-effectivity.
In \cite{BKKMSU15}, these positivities were discussed by using the base loci of vector bundles  (see also \cite{FN23}).
Note that a pseudo-effective (resp.~big) vector bundle in this paper 
is said to be {\em{V-psef}} (resp.~{\em{V-big}}) in \cite[Definition~2.2]{FN23}. 

Theorem~\ref{thm:negative} also solves another problem posed by Demailly, Peternell, and Schneider~\cite[~Problem 6.6]{DPS01}. They introduced the notion of almost nefness for vector bundles (see Definition~\ref{defn:almostnef}) as a generalization of nefness, and proved that a pseudo-effective vector bundle is almost nef (\cite[~Proposition 6.5]{DPS01}), leaving the converse as a problem (\cite[~Problem 6.6]{DPS01}). I.e., they asked whether almost nefness implies pseudo-effectivity. 
This is solved negatively by Theorem~\ref{thm:negative}, because the preservation of almost nefness by extensions implies that $\mathcal V$ in the theorem is almost nef but not pseudo-effective. 

\begin{ack}
The authors thank Mihai Fulger, Shin-ichi Matsumura, Niklas M\"uller, and Xiaojun Wu very much for some useful comments and suggestions.
They also thank the referee for helpful suggestions. 
The first author was partly supported by MEXT Promotion of Distinctive Joint Research Center Program JPMXP0619217849.
The second author was partially supported by JSPS KAKENHI Grant Numbers JP19H01787, JP20H00111, JP21H00974, JP21H04994.
The third author was supported by Grant-in-Aid for Early Career Scientists JP22K13907.
\end{ack}
\section{Definitions}

In this section, we recall several definitions defined for vector bundles.
Let $k$ be an algebraically closed field of arbitrary characteristic.
A {\it variety} is an integral separated scheme of finite type over $k$.
\begin{defn} \label{defn:ample_nef}
Let $\mathcal E$ be a vector bundle on a projective variety $X$. 
Let $\pi:\mathbb P(\mathcal E) \to X$ be the projectivization of $\mathcal E$. 
Let $\mathcal O_{\mathbb P(\mathcal E)}(1)$ be the tautological line bundle. 
We say that $\mathcal E$ is \textit{ample} (resp.~\textit{nef}) if $\mathcal O_{\mathbb P(\mathcal E)}(1)$ is ample (resp.~nef). 
\end{defn}
\begin{defn} \label{defn:ggg}
Let $\mathcal G$ be a coherent sheaf on a variety $X$. 
Let $U$ be an open subset of $X$. 
We say that $\mathcal G$ is \textit{globally generated over $U$} (resp.~\textit{generically globally generated})
if the natural map 
$$
H^0(X,\mathcal G)\otimes_k \mathcal O_X \to \mathcal G
$$
is surjective over $U$ (resp.~surjective at the generic point of $X$). 
\end{defn}
\begin{defn}[\textup{\cite[Definition~1.2]{Vie82}}] \label{defn:psef_wp}
Let $\mathcal G$ be a vector bundle on 
a quasi-projective variety $X$. 
Let $U$ be an open subset of $X$. 
Let $H$ be an ample Cartier divisor on $X$. 
We say that $\mathcal G$ is \textit{weakly positive over $U$} 
(resp.~\textit{pseudo-effective}) 
if for every $\alpha\in\mathbb Z_{>0}$, there exists a $\beta\in\mathbb Z_{>0}$
such that 
$
S^{\alpha\beta}(\mathcal G) (\beta H)
$
is globally generated over $U$ (resp.~generically globally generated). 
Here, $S^{\alpha\beta}(\mathcal G)$ denotes 
the $\alpha\beta$-th symmetric product of $\mathcal G$. 
We say that $\mathcal G$ is \textit{weakly positive} 
if $\mathcal G$ is weakly positive over an open subset of $X$. 
\end{defn}
We further assume that $X$ is projective in Definition \ref{defn:psef_wp}.
In this case, we can easily check that $\mathcal G$ is nef if and only if $\mathcal G$ is weakly positive over $X$.
\begin{defn}[\textup{\cite[Notation~(vii)]{Kol87}}] \label{defn:big}
Let $\mathcal G$ be a vector bundle on 
a quasi-projective variety $X$. 
Let $H$ be an ample Cartier divisor on $X$. 
We say that $\mathcal G$ is \textit{big} if there exists an 
$\alpha \in\mathbb Z_{>0}$ such that 
$
S^\alpha (\mathcal G) (-H)
$
is pseudo-effective. 
\end{defn}
By \cite[Lemma~2.14]{Vie95}, Definitions~\ref{defn:psef_wp} and~\ref{defn:big} are independent of the choice of ample Cartier divisor $H$. 
Note that, for a generically surjective morphism $\mathcal E \to \mathcal F$ between vector bundles, if $\mathcal E$ is pseudo-effective (resp.~weakly positive, big), then so is $\mathcal F$.
\begin{rem}
The terminology ``pseudo-effective'' (resp.~``big'') is often used 
in a different meaning. 
For example, in other papers, a vector bundle $\mathcal E$ on 
a projective variety $X$ is said to be pseudo-effective (resp.~big) if 
$\mathcal{O}_{\mathbb{P}(\mathcal E) }(1)$ is pseudo-effective (resp.~big). 
This is weaker than the pseudo-effectivity (resp.~bigness) in this paper. 
\end{rem}
\begin{rem}
To the best knowledge of the authors, the notion of 
{\em{weakly positive sheaves}} 
was first introduced by Viehweg in \cite{Vie82} 
(see \cite[Definition~1.2]{Vie82}) and that of 
{\em{big sheaves}} originates from 
\cite{Vie83II} (see \cite[Lemma~3.6]{Vie83II}). 
We note that the definition of weak positivity in 
\cite{Vie83II} is different from the one in \cite{Vie82} 
(see also \cite[Definition~1.2]{Vie83}) 
and coincides with that of {\em{pseudo-effectivity}}.  
\end{rem}
\begin{defn}
\label{defn:almostnef}
Let $\mathcal E$ be a vector bundle  on a projective variety $X$.
We say that $\mathcal E$ is \textit{almost nef} if there exists a countable family $A_i$ 
of proper subvarieties of $X$ such that $\mathcal E|_C$ is nef for all curves 
$C \not\subset \bigcup_i A_i$. 
\end{defn}
\section{Proof of Theorem~\ref{thm:nef-psef}}
Before starting the proof of Theorem~\ref{thm:nef-psef}, 
we recall the following lemma.
\begin{lem}[\textup{\cite[Lemma~2.15]{Vie95}}] \label{lem:Viehweg}
Let $X$ be a smooth projective variety.
Let $\mathcal E$ be a vector bundle on $X$. 
Then $\mathcal E$ is pseudo-effective if and only if for every finite surjective morphism $\pi:X'\to X$ from a smooth projective variety $X'$ and for every ample divisor $H'$ on $X'$, the vector bundle $\pi^*\mathcal E(H')$ is pseudo-effective. 
\end{lem}
\begin{proof}[Proof of Theorem~\ref{thm:nef-psef}.]
First, we prove (1). 
When $\mathrm{char}(k)=0$ (resp.~$\mathrm{char}(k)>0$), 
we take a resolution of singularities (resp.~a smooth 
alteration constructed in \cite[4.1.~Theorem]{deJ96}). 
Then, by \cite[Lemma~2.4~(2)]{EG19}, we may assume that 
$X$ is smooth. 
We replace $X$ with any finite cover of $X$.  
By Lemma \ref{lem:Viehweg}, it is enough to show that 
$\mathcal E(H)$ is pseudo-effective for every ample divisor $H$ on $X$. 
Note that $\mathcal L(H)$ is big. 
When $\mathrm{char}(k)=0$, by taking a resolution of 
a suitable cyclic cover 
and using \cite[Lemma~2.4~(2)]{EG19}, 
we may assume that there is an injective 
morphism $\mathcal O_X \hookrightarrow \mathcal L(H)$. 
When $\mathrm{char}(k)>0$, by using the Frobenius morphism and 
\cite[Lemma~2.4~(2)]{EG19},  
we can replace $\mathcal L(H)$ by $\mathcal L(H)^{p^e}$ and 
obtain an injective morphism $\mathcal O_X \hookrightarrow \mathcal L(H)$. 
Let $\mathcal F$ be the inverse image of $\mathcal O_X$ by 
$\mathcal E(H) \twoheadrightarrow \mathcal L(H)$. 
Then we have the morphism between exact sequences
$$
\xymatrix{
0 \ar[r] & \mathcal G(H) \ar[r] \ar@{=}[d] & \mathcal F \ar[r] \ar@{^(->}[d]^\tau & \mathcal O_X \ar[r] \ar@{^(->}[d] & 0 \\
0 \ar[r] & \mathcal G(H) \ar[r] & \mathcal E(H) \ar[r] & \mathcal L(H) \ar[r] & 0.
}
$$
Since $\mathcal G(H)$ and $\mathcal O_X$ are nef, 
we see that $\mathcal F$ is also nef. 
By the generic surjectivity of $\tau$, 
we see that $\mathcal E(H)$ is pseudo-effective. 

Next, we prove (2). 
Let $H$ be an ample Cartier divisor on $X$. 
Take $m\in\mathbb Z_{>0}$ such that $\mathrm{char}(k)\nmid m$, $S^m(\mathcal G)(-H)$ is nef, and $\mathcal L^m(-H)$ is pseudo-effective.
By \cite[Theorem~4.1.10]{Laz04I}, there are a surjective finite morphism 
$\pi:X'\to X$ from a normal projective variety $X'$ 
and an ample Cartier divisor $H'$ on $X'$ such that $mH'\sim \pi^*H$. 
Consider the exact sequence 
$$
0\to \pi^*\mathcal G(-H') \to \pi^*\mathcal E(-H') 
\to \pi^*\mathcal L(-H') \to 0. 
$$
Since $S^m(\pi^*\mathcal G(-H')) \cong \pi^*\big(S^m(\mathcal G)(-H)\big)$ 
is nef, so is $\pi^*\mathcal G(-H')$. 
Also, $\pi^*\mathcal L(-H')$ is pseudo-effective, 
so we see that 
$\pi^*\mathcal E(-H')$ is pseudo-effective by (1). 
Then 
\begin{align*}
S^{(m+1)\beta}(\pi^*\mathcal E(-H'))(\beta H')
& \cong S^{(m+1)\beta}(\pi^*\mathcal E) (-m\beta H')
\\ & \cong S^{(m+1)\beta}(\pi^*\mathcal E) (-\beta \pi^*H)
\cong \pi^*\left( S^{(m+1)\beta}(\mathcal E)(-\beta H) \right)
\end{align*}
is generically globally generated for some $\beta \in \mathbb Z_{>0}$, 
so $S^{(m+1)\beta}(\mathcal E)(-\beta H)$ 
is pseudo-effective by \cite[Lemma~2.4~(2)]{EG19}, 
which means that $\mathcal E$ is big. 
\end{proof}
If the following question is answered affirmatively, 
then we can generalize Theorem~\ref{thm:nef-psef} to the case of higher rank 
by an argument similar to that of the above proof. 
\begin{ques}
Let $\mathcal E$ be a \textit{big} vector bundle 
on a normal projective variety $X$ over an algebraically closed field. 
Does there exist a surjective finite morphism $\pi:X'\to X$ from 
a normal projective variety $X'$ such that 
$\pi^*\mathcal E$ is generically globally generated? 
\end{ques}
For example, one can easily check that the question holds affirmatively 
if $\mathcal E$ is a direct sum of big line bundles. 
\section{Proof of Theorem~\ref{thm:negative}}
Set $X:=\mathbb P(\mathcal O_{\mathbb P^1} \oplus \mathcal O_{\mathbb P^1}(-2))$. 
Let $f:X\to \mathbb P^1$ be the projection. 
Let $C\subset X$ be the section of $f$ corresponding to the quotient $\mathcal O_{\mathbb P^1} \oplus \mathcal O_{\mathbb P^1}(-2) \twoheadrightarrow \mathcal O_{\mathbb P^1}(-2)$. 
Then $\mathcal O_X(1) \cong \mathcal O_X(C)$. 
We define the  divisor $H$ on $X$ as $$H:=C +3f^* [y],$$ where $y \in \mathbb P^1$ is a closed point. 
Then we see from \cite[V, Theorem~2.17]{Har77} that $H$ is very ample. 

Since 
$
f_*\mathcal O_X(C) \cong \mathcal O_{\mathbb P^1} \oplus \mathcal O_{\mathbb P^1}(-2), 
$
we have $H^1(\mathbb P^1, f_*\mathcal O_X(C)) \cong k$. 
Thus, from the Leray spectral sequence, we obtain 
$H^1(X, \mathcal O_X(C))\cong k.$
Take $0\ne \xi \in \mathrm{Ext}^1(\mathcal O_X, \mathcal O_X(C))$.
Let 
\begin{align} 
\label{align:xi}\tag{$\flat$}
0\to \mathcal O_X(C) \to \mathcal E \to \mathcal O_X \to 0
\end{align}
be the exact sequence corresponding to $\xi$. 
Since $H^1(X,\mathcal O_X)=0$, we see that the natural morphism 
$$
H^1(X, \mathcal O_X(C))
\to 
H^1(C, \mathcal O_C(C))
$$
is injective, so the exact sequence 
$$
0 \to \mathcal O_C(C) \to \mathcal E|_C \to \mathcal O_C \to 0
$$
does not split. Since $\mathcal O_C(C) \cong \mathcal O_{\mathbb P^1}(-2)$, 
we see that 
$
\mathcal E|_C \cong \mathcal O_{\mathbb P^1}(-1)^{\oplus 2}.
$ 

From now on, we divide the proof into the case of $\mathrm{char}(k)=0$ 
and the case of $\mathrm{char}(k)=p>0$. 
\medskip 
\\ \noindent \underline{Case of $\mathrm{char}(k)=0$}.  
By \cite[Theorem~4.1.10]{Laz04I}, there is a surjective finite morphism 
$\pi:X'\to X$ from a smooth projective surface 
and an ample Cartier divisor $H'$ on $X$ such that $\pi^*H \sim 4H'$. 
Put $\mathcal G:= \pi^*\mathcal E (H')$. 
Taking the pullback of $($\ref{align:xi}$)$ and the tensor product with $\mathcal O_{X'}(H')$, 
we obtain the exact sequence 
$$
0\to \mathcal O_{X'}(\pi^*C +H') \to \mathcal G \to \mathcal O_{X'}(H') \to 0. 
$$
We prove that $\mathcal G$ is not pseudo-effective. 
From 
$$
S^4(\mathcal G) 
\cong S^4(\pi^*\mathcal E) (4H')
\cong S^4(\pi^*\mathcal E) (\pi^*H)
\cong \pi^*\left( S^4(\mathcal E) (H) \right),
$$ 
it is enough to show that $\mathcal F:=S^4(\mathcal E) (H)$ is not pseudo-effective by \cite[Corollary~2.20]{Vie95} and \cite[Lemma~2.4~(2)]{EG19}.
For this purpose, we check that 
$$
S^{4\beta}(\mathcal F)(\beta H)
\cong S^{4\beta}\left(S^4(\mathcal E)\right)(5\beta H)
$$ 
is not generically globally generated for each $\beta \in \mathbb Z_{>0}$. 
By $($\ref{align:xi}$)$, we have the following surjective morphism 
$$
\sigma_\beta:
S^{4\beta}\left(S^4(\mathcal E)\right)(5\beta H)
\twoheadrightarrow
\mathcal O_X(5\beta H). 
$$
Let us consider the following commutative diagram: 
\begin{align*} 
\xymatrix{
S^{4\beta}\left(S^4(\mathcal E)\right)(15 \beta f^*[y]) 
\ar@{^(->}[r]^-{\tau_\beta} 
\ar@{->>}[d]_-{\lambda_\beta} 
&
S^{4\beta}\left(S^4(\mathcal E)\right) (5\beta H) \ar@{->>}[d]^-{\sigma_\beta}
\\ 
\mathcal O_X(15\beta f^*[y]) \ar@{^(->}[r]
&
\mathcal O_X(5\beta H). 
}
\end{align*}
Here, the horizontal arrows are induced from the morphism 
$
\mathcal O_X(-5\beta C) \hookrightarrow \mathcal O_X. 
$
In order to prove that $S^{4\beta}\left(S^4(\mathcal E)\right)(5\beta H)$ is not generically globally generated, it is enough to see that $H^0(\sigma_\beta)$ is the zero-map. 
For this purpose, it is sufficient to prove $H^0(\tau_\beta)$ is bijective and 
$H^0(\lambda_\beta)$ is the zero-map.
\begin{cl} \label{cl:3}
$H^0(\tau_\beta)$ is bijective. 
\end{cl}
\begin{proof}[Proof of Claim \ref{cl:3}]
 Taking the tensor product of 
$$
0 \to \mathcal O_X(-C) \to \mathcal O_X \to \mathcal O_C \to 0
$$
and $S^{4\beta}\left(S^4(\mathcal E)\right)\left( lC +15\beta f^*[y]\right)$ for $l\in\mathbb Z_{\ge 0}$, 
we obtain the following exact sequence: 
\begin{align*}
0 & \to H^0\left(X, S^{4\beta}\left(S^4(\mathcal E)\right) \left((l-1)C +15\beta f^*[y]\right) \right)
\\ & \to H^0\left(X, S^{4\beta}\left(S^4(\mathcal E)\right) \left(lC +15\beta f^*[y]\right) \right)
\to H^0\left(C, S^{4\beta}\left(S^4(\mathcal E)\right) \left(lC +15\beta f^*[y]\right) |_C \right). 
\end{align*}
Since $\mathcal O_C(C) \cong \mathcal O_{\mathbb P^1}(-2)$ and 
$\mathcal E|_C \cong \mathcal O_{\mathbb P^1}(-1)^{\oplus 2}$, 
we have 
\begin{align*}
H^0\left(C, S^{4\beta}\left(S^4(\mathcal E)\right) \left(lC +15 \beta f^*[y]\right) |_C\right)
& \cong  \bigoplus H^0\left(\mathbb P^1, \mathcal O_{\mathbb P^1}(-16\beta -2l +15\beta) \right)
\\ & = \bigoplus H^0\left(\mathbb P^1, \mathcal O_{\mathbb P^1}(-\beta -2l) \right)
=0. 
\end{align*}
From  $H^0\left(X, S^{4\beta}\left(S^4(\mathcal E)\right) (5\beta H) \right) = H^0\left(X, S^{4\beta}\left(S^4(\mathcal E)\right) (5\beta C  + 15 \beta f^*[y] )\right) $, our claim follows.
\end{proof}
\begin{cl} \label{cl:4}
$H^0(\lambda_\beta)$ is the zero-map.	
\end{cl}
\begin{proof}[Proof of Claim \ref{cl:4}]
Consider the following commutative diagram: 
$$
\xymatrix{
H^0\left(X, S^{4\beta}\left(S^4(\mathcal E)\right) (15\beta f^*[y]) \right) 
\ar[r] \ar[d]_{H^0(\lambda_\beta)} 
& H^0\left(C, S^{4\beta}\left(S^4(\mathcal E)\right) (15\beta f^*[y]) |_C \right) \ar[d] 
\\ H^0\left(X, \mathcal O_X(15\beta f^*[y]) \right) \ar[r] 
& H^0\left(C, \mathcal O_C(15\beta f^*[y]) \right). 
}
$$
The bottom horizontal arrow is bijective, since $C$ is a section of $f$. 
Hence, our claim follows from 
$$
H^0\left(C, S^{4\beta}(S^4(\mathcal E))(15\beta f^*[y]) |_C \right) 
\cong 
\bigoplus H^0\left(\mathbb P^1, \mathcal O_{\mathbb P^1}(-16\beta +15\beta) \right)
=0.
$$
\end{proof}
We put $S := X'$, $\mathcal V := \mathcal G$, $\mathcal L := \mathcal O_{X'}(\pi^*C +H') $, and $\mathcal M :=  \mathcal O_{X'}(H')$. 
Then they satisfy all the desired properties.

\medskip 
\noindent \underline{Case of $\mathrm{char}(k)=p>0$}.  
Set $e:=1$ (resp.~$e:=2$) if $p\ge 5$ (resp.~$p<5$). 
Then $p^e\ge 4$. 
Put $\mathcal G:=({F^e}^*\mathcal E)(H)$, 
where $F$ is the absolute Frobenius morphism of $X$. 
Taking the pullback of  $($\ref{align:xi}$)$ by ${F^e}$ and  the tensor product with $\mathcal O_X(H)$, 
we obtain the exact sequence 
$$
0\to \mathcal O_X(p^{e}C+H) \to \mathcal G \to \mathcal O_X(H) \to 0. 
$$
We prove that $\mathcal G$ is not pseudo-effective. 
For this purpose, we check that 
$$
S^{4\beta}(\mathcal G)(\beta H)
\cong 
S^{4\beta}({F^e}^*\mathcal E)(5\beta H)
$$
is not generically globally generated for each $\beta\in\mathbb Z_{>0}$. 
We have the following surjective morphism 
$$
s_\beta: S^{4\beta}({F^e}^*\mathcal E)(5\beta H)
\twoheadrightarrow \mathcal O_X(5\beta H).
$$
Thus, it is enough to check that $H^0(s_\beta)$ is the zero-map. 
For each $l\in\mathbb Z_{\ge 0}$, we have 
\begin{align*}
H^0\left(C, S^{4\beta}({F^e}^*\mathcal E)(lC+15\beta f^*[y]) |_C \right)
&\cong  \bigoplus H^0\left(\mathbb P^1, \mathcal O_{\mathbb P^1}(-4\beta p^e -2l +15\beta) \right)
\\ & \cong  \bigoplus H^0\left(\mathbb P^1, \mathcal O_{\mathbb P^1}\left ((15-4p^e)\beta -2l \right) \right)
 =0. 
\end{align*}
Note that $15-4p^e\le 15-16=-1$. 
Hence, we can prove $H^0(s_\beta)=0$ by an argument similar to that of $H^0(\sigma_\beta)=0$ as in the  $\mathrm{char}(k)=0$ case. 

We put $S := X$, $\mathcal V := \mathcal G$, $\mathcal L := \mathcal O_X(p^{e}C+H) $, and $\mathcal M :=  \mathcal O_X(H)$. 
Then they satisfy all the desired properties.
\qed
\begin{rem}
The vector bundle $\mathcal E$ in $($\ref{align:xi}$)$ is a simple example of an almost nef but not pseudo-effective vector  bundle. That $\mathcal E$ is not pseudo-effective is proved implicitly in the proof above, but can also be proved directly.
By $($\ref{align:xi}$)$, we get the surjective morphism  
$$
t_\beta: S^{4\beta}(\mathcal E)(\beta H) 
\twoheadrightarrow \mathcal O_X(\beta H)
$$
for each $\beta \in \mathbb Z_{>0}$. 
We can prove $H^0(t_\beta)=0$ 
by using the commutative diagram 
\begin{align*} 
\xymatrix{
S^{4\beta}(\mathcal E)(3 \beta f^*[y]) 
\ar@{^(->}[r]
\ar@{->>}[d]
&
S^{4\beta}(\mathcal E)(\beta H) \ar@{->>}[d]^-{t_\beta}
\\ 
\mathcal O_X(3\beta f^*[y]) \ar@{^(->}[r]
&
\mathcal O_X(\beta H),  
}
\end{align*}
where the horizontal arrows are induced from the morphism $\mathcal O_X(-\beta C) \hookrightarrow \mathcal O_X$, 
and a vanishing as in Claim~\ref{cl:3}, that is, 
$$
H^0\left(C, S^{4\beta}(\mathcal E)(lC+3\beta f^*[y]) |_C\right) 
=\bigoplus H^0(\mathbb P^1, \mathcal O_{\mathbb P^1}(-4\beta -2l +3\beta))
=0 
$$
for each $l\in\mathbb Z_{\ge 0}$. 
\end{rem}
\begin{rem}
In positive characteristic, we do not know whether the
pseudo-effectivity of $\mathcal E$ implies that of $S^m(\mathcal E)$,
so we choose to separate the proof into the case of $\mathrm{char}(k)=0$ and the
case of $\mathrm{char}(k)>0$.
\end{rem}
\bibliographystyle{abbrv}
\bibliography{ref.bib}

\begin{thebibliography}{10}

\bibitem{BKKMSU15}
T.~Bauer, S.~J. Kov{\'a}cs, A.~K{\"u}ronya, E.~C. Mistretta, T.~Szemberg, and
  S.~Urbinati.
\newblock On positivity and base loci of vector bundles.
\newblock {\em European Journal of Mathematics}, 1(2):229--249, 2015.

\bibitem{CP91}
F.~Campana and T.~Peternell.
\newblock Projective manifolds whose tangent bundles are numerically effective.
\newblock {\em Math. Ann.}, 289(1):169--187, 1991.

\bibitem{deJ96}
A.~J. de~Jong.
\newblock Smoothness, semi-stability and alterations.
\newblock {\em Publ. Math. Inst. Hautes \'Etudes Sci.}, 83(1):51--93, 1996.

\bibitem{DPS94}
J.-P. Demailly, T.~Peternell, and M.~Schneider.
\newblock Compact complex manifolds with numerically effective tangent bundles.
\newblock {\em J. Algebraic Geom.}, 3(2):295--346, 1994.

\bibitem{DPS01}
J.-P. Demailly, T.~Peternell, and M.~Schneider.
\newblock Pseudo-effective line bundles on compact {K}{\"a}hler manifolds.
\newblock {\em Int. J. Math.}, 12(06):689--741, 2001.

\bibitem{EG19}
S.~Ejiri and Y.~Gongyo.
\newblock Nef anti-canonical divisors and rationally connected fibrations.
\newblock {\em Compos. Math.}, 155(7):1444--1456, 2019.

\bibitem{FN23}
M.~Fulger and N.~Ray.
\newblock Positivity and base loci for vector bundles revisited.
\newblock {\em arXiv preprint arXiv:2303.13201}, 2023.
\newblock to appear in Tohoku Math. J.

\bibitem{Har70}
R.~Hartshorne.
\newblock Ample subvarieties of algebraic varieties.
\newblock {\em Lecture Notes in Mathematics}, 156, 1970.

\bibitem{Har77}
R.~Hartshorne.
\newblock {\em Algebraic Geometry}.
\newblock Number~52 in Grad. Texts in Math. Springer-Verlag New York, 1977.

\bibitem{HMI22}
G.~Hosono, M.~Iwai, and S.-i. Matsumura.
\newblock On projective manifolds with pseudo-effective tangent bundle.
\newblock {\em J. Inst. Math. Jussieu}, 21(5):1801--1830, 2022.

\bibitem{Kol87}
J.~Koll{\'a}r.
\newblock Subadditivity of the {K}odaira dimension: fibers of general type.
\newblock {\em Adv. {S}tud. in {P}ure {M}ath.}, 10:361--398, 1987.

\bibitem{Laz04I}
R.~K. Lazarsfeld.
\newblock {\em Positivity in Algebraic Geometry {I}}, volume~48 of {\em
  Ergebnisse der {M}athematik und ihrer {G}renzgebiete. 3. {F}olge}.
\newblock Springer-Verlag Berlin Heidelberg, 2004.

\bibitem{Laz04II}
R.~K. Lazarsfeld.
\newblock {\em Positivity in Algebraic Geometry {I\hspace{-1pt}I}}, volume~49
  of {\em Ergebnisse der {M}athematik und ihrer {G}renzgebiete. 3. {F}olge}.
\newblock Springer-Verlag Berlin Heidelberg, 2004.

\bibitem{Mor79}
S.~Mori.
\newblock Projective manifolds with ample tangent bundles.
\newblock {\em Ann. of Math. (2)}, 110(3):593--606, 1979.

\bibitem{Vie82}
E.~Viehweg.
\newblock Die {A}dditivit{\"a}t der {K}odaira {D}imension f{\"u}r projektive
  {F}aserr{\"a}ume {\"u}ber {V}ariet{\"a}ten des allgemeinen {T}yps. ({G}erman)
  [the additivity of the {K}odaira dimension for projective fiber spaces over
  varieties of general type].
\newblock {\em J. Reine Angew. Math.}, 330:132--142, 1982.

\bibitem{Vie83}
E.~Viehweg.
\newblock Weak positivity and the additivity of the {K}odaira dimension for
  certain fiber spaces.
\newblock In {\em Algebraic Varieties and Analytic Varieties}, pages 329--353.
  Kinokuniya, North-Holland, 1983.

\bibitem{Vie83II}
E.~Viehweg.
\newblock Weak positivity and the additivity of the {K}odaira dimension
  {I\hspace{-1pt}I}: the local {T}orelli map.
\newblock In {\em Classification of Algebraic and Analytic Manifolds, Katata
  1982}, volume~39 of {\em Progr. in Math.}, pages 567--589. Birkh\"{a}user,
  1983.

\bibitem{Vie95}
E.~Viehweg.
\newblock {\em Quasi-projective Moduli for Polarized Manifolds}, volume~30 of
  {\em Ergebnisse der {M}athematik und ihrer {G}renzgebiete. 3. {F}olge}.
\newblock Springer, Berlin, Heidelberg, 1995.

\end{thebibliography}
\end{document}